\documentclass[11pt]{amsart}

\usepackage{mathrsfs}
\usepackage[dvipdfmx]{graphicx}
\usepackage{epsfig}
\usepackage{amsmath}
\usepackage{amssymb}
\usepackage{amscd}
\usepackage{latexsym}
\usepackage{tabularx}
\usepackage{a4wide}
\usepackage[usenames]{color}
\usepackage{enumerate}
\usepackage{subfigure}
\usepackage{url}
\usepackage{xspace}

\newtheorem{thm}{Theorem}[section]

\newtheorem{lemma}[thm]{Lemma}

\newtheorem{prop}[thm]{Proposition}

\newtheorem{cor}[thm]{Corollary}

\newtheorem{remark}[thm]{Remark}

\newtheorem{defi}[thm]{Definition}

\newcommand {\mm}[1] {\ifmmode{#1}\else{\mbox{\(#1\)}}\fi}

\newcommand{\nr}{\mathrm{nr}}

\newcommand{\R}        {\mm{{\mathbb R}}}

\newcommand{\U}            {\mathcal U}
\newcommand{\vol}             {\mathrm{vol}}

\newcommand{\dm}             {d_{max}}

\newcommand{\Zh}{\mathrm{Zh}}

\newcommand{\g}{\mathfrak g}
\newcommand{\h}{\mathfrak h}

\title{A transfer principle and applications to eigenvalue
  estimates for graphs}
\author{Omid Amini}
\address{CNRS - DMA, \'Ecole Normale Sup\'erieure, Paris}
 \email{oamini@math.ens.fr}
 \author{David Cohen-Steiner}
 \address{INRIA, 2004 Route des Lucioles, BP93,
           Sophia-Antipolis, France}
\email{david.cohen-steiner@inria.fr}

\begin{document}
\maketitle

\begin{abstract}
In this paper, we prove a variant of the Burger-Brooks transfer principle which, combined
 with recent eigenvalue bounds for surfaces, allows to obtain upper
 bounds on the eigenvalues of graphs as a function of their
 genus. More precisely, we show the existence of a universal constants
 $C$ such that the $k$-th  eigenvalue $\lambda_k^{nr}$ of the
 normalized Laplacian of a graph $G$ of (geometric) genus $g$ on $n$
 vertices satisfies $$\lambda_k^{nr}(G)  \leq C \frac{d_{\max}(g+k)}{n},$$ 
where $d_{\max}$ denotes the maximum valence of vertices of the
graph. This result is tight up to a change in the value of the constant $C$, and improves recent results of Kelner, Lee, Price and Teng on bounded genus graphs.
 
 To show that the transfer theorem might be of independent interest, we relate eigenvalues of the
 Laplacian on a metric graph to the eigenvalues of its 
 simple graph models, and discuss an application to the
 mesh partitioning problem, extending pioneering results of  Miller-Teng-Thurston-Vavasis and Spielman-Tang to arbitrary meshes.
\end{abstract}

\section{Introduction}
The spectrum of the Laplacian of a finite graph reflects information about the structural properties of the graph and has been
successfully used in a large variety of applications to other domains.
In particular, the eigenvalues of a bounded degree graph provide information on the existence of
good clusterings of that graph, see~\cite{AM85} for clustering in two classes and \cite{LGT12,LRTV12} for $k$-way clusterings, whose optimal quality 
is shown to relate to the $k$-th eigenvalue. 
 
\noindent In particular, upper bounds on the
eigenvalues of a class of graphs directly translate into efficient clustering
algorithms with quality guarantees. This motivated a series of work,
starting with Spielman and Teng \cite{ST07}, who gave an $O(1/n)$ bound for the
Fiedler value of a bounded degree planar graph on $n$ vertices, using
a suitably centered circle packing representation of the graph.  Kelner
extended this result to an $O((g+1)/n)$ bound for (geometric) genus $g$ graphs \cite{K06}. The
argument  uses Riemann-Roch
theorem to find a circle packing representation of the graph.  Recently, Kelner, Lee, Price and Teng
proved an $O((g+1)\log(g+1)^2k/n)$ upper bound for the $k$-th
eigenvalue \cite{KLPT11}, using a multicommodity flow problem to suitably uniformize the graph metric.

The study of the spectrum of a finite graph is in many
  ways related to the spectral theory of Riemannian manifolds, and results in geometric analysis have been a source of inspiration to state and 
prove corresponding results concerning finite graphs. In particular, eigenvalue bounds for surfaces have a somewhat parallel history. Hersch \cite{H70} first proved an $O(1/\vol(M))$ bound for the Neumann value of the sphere $\mathbb S^2$ 
equipped with a Riemannian metric. Yang and Yau \cite{YY80}
then showed that for genus $g$ surfaces an $O((g+1)/\vol(M))$ bound holds,
and Li and Yau  improved the latter result by replacing the genus with
the finer conformal invariant they defined~\cite{LY82}. It is interesting to notice that these 
proofs are quite similar at a high level to the ones later used in the graph
setting.  Conformal uniformization was used in place of circle packing
representations, but the very same topological argument for
centering the packing in the discrete case was used in the manifold
case as well. For higher eigenvalues, Korevaar \cite{K93} established an
$O((g+1)k/\vol(M))$ for genus $g$ surfaces, and Hassannezhad \cite{H11} improved
this to $O((g+k)/\vol(M))$, by combining the two methods of constructing disjoint
capacitors of Grigor'yan, Netrusov and Yau \cite{GNY04}, and Colbois and Maerten \cite{CM08}.

While traditionally bounds on graph eigenvalues are 
used to prove bounds for Riemannian manifolds~\cite{Br86, BM04, B86, B78, CC88, C86},  it is intriguing
to see that the spectral theory of Riemannian manifolds has not
been much used so far to provide information on the spectral
properties of general finite graphs.

Our aim in this paper is to show how 
 eigenvalue bounds for surfaces combined with basic spectral theory of (singular) surfaces, and 
 a  suitable transfer principle, 
 allows to obtain eigenvalue estimates for graphs in terms of their geometric genus. In this way, we are able to 
 extend the above mentioned result of~\cite{H11, K93, YY80} for surfaces to the graph setting using a suitable variant of
the Burger-Brooks transfer method, c.f. Theorem~\ref{thm:2-fold}. Our results are tight and improve the recent results of Kelner, Lee, Price and Teng~\cite{KLPT11} on bounded genus graphs. In addition to providing a uniform arguably more conceptual proof of the results of~\cite{KLPT11, K06, ST07}, we hope that our method makes the above mentioned existing similarities between the methods used in the spectral theory of surfaces and graphs more transparent.

 The transfer principle proved in this paper may be of independent interest. We shall show that it can be used to provide uniform upper and lower bounds on the 
 eigenvalues of metric graphs in terms of the eigenvalues of their simple graph models. Furthermore, it allows to generalize to completely arbitrary meshes the mesh partitioning results of Miller, Teng, Thurston, and Vavasis~\cite{MTTV98} and Spielman and Teng~\cite{ST07}.

\subsection{Statement of the main theorem on eigenvalues of bounded genus graphs} 
Let $G=(V,E)$ be a finite simple graph, that we assume connected all through the paper. For two vertices $u,v \in V$, we write $u \sim v$ if the two vertices $u$
and $v$ are connected by an edge in $G$. The valence of a vertex $v$ of $G$ is denoted by $d^G_v$, or simply $d_v$ if there is no risk of confusion and the graph $G$ is understood from the context. We denote by  
$d_{\max}$ the maximum degree of vertices of the graph, and by $n$ the number of vertices. 
The geometric genus of $G$ is by definition the minimum integer $g$ such that $G$ can be embedded with no crossing on 
the compact orientable surface of genus $g$.

Denote by $C(G)$ the vector space of all real valued functions $f$ defined on the set of vertices of $G$.
The (discrete) Laplacian $\Delta$ and the normalized Laplacian $\mathcal L$ of $G$ are defined as follows:  the Laplacian  $\Delta: C(G) \rightarrow C(G)$ is the linear operator which sends a function $f\in C(G)$ to $\Delta(f)\in C(G)$ defined   by $$\forall\, v\in V(G), \quad \Delta(f)(v) = \sum_{u: u\sim v} f(v) -f(u).$$

Let $S$ be the linear operator on $C(G)$ whose matrix in the standard basis of $C(G)$ is diagonal with entries the valences of the vertices of
$G$, i.e., for any $f\in C(G)$
\[ \forall\, v\in V(G), \quad S(f)(v) = d_v f(v).\] 
The normalized Laplacian is the operator $S^{-1/2}\Delta
S^{-1/2}$.

\medskip

We denote by 
$$\lambda_0(G) = 0 < \lambda_1(G) \leq \lambda_2(G) \leq \dots \leq
\lambda_{n-1}(G)$$ the set of eigenvalues of $\Delta$, which we call the
standard spectrum of $G$, and by $$\lambda_0^{\nr}(G) =0 <
\lambda_1^{\nr}(G) \leq \dots \leq \lambda_{n-1}^{\nr}(G)$$ the set of all
eigenvalues of the normalized Laplacian $\mathcal L$, which we call
the normalized spectrum. The standard and normalized spectrum of $G$ are easily seen to satisfy the inequalities 
$d_{\min} \,\lambda_k^{\nr}(G) \leq \lambda_k(G) \leq d_{\max} \,\lambda_k^{\nr}(G)$
for any $k$. 

\medskip

In this paper we prove the following theorem.

\begin{thm}\label{thm:main} There exists a universal constant $C$ such that the eigenvalues of the
normalized Laplacian of any graph $G$ on $n$ vertices satisfy:
$$
\forall k\in \mathbb N, \qquad \lambda_k^{\nr}(G) \leq C\;\frac{\dm (g+k)}{n},
$$
where $d_{\max}$ and $g$ are the maximum valence and the geometric
genus of $G$, respectively.
\end{thm}

The linear dependance in the maximum degree is clearly optimal, as can
be seen by considering star graphs, which have lower bounded Fiedler
value. The above result also implies a similar bound for the eigenvalues of the standard
Laplacian, at the expense of an extra $\dm$ factor. We note that
Kelner, Lee, Price and Teng~\cite{KLPT11} give a similar bound for
the standard spectrum with a linear rather than quadratic dependence
in $\dm$. However, their bound has a $gk \log(g+1)^2$ dependence
instead of our $(g+k)$ dependence. 
In addition to simplifying and improving the result of~\cite{KLPT11} 
for bounded genus graphs, we note that the dependence in $g$ and $k$ of our estimate is
tight, at least when $g$ is sufficiently high, see Remark~\ref{rem:tightness}.

Informally, the improvement over \cite{KLPT11} means
that the asymptotic behavior of graphs' 
eigenvalues do not depend on the (geometric) genus of the graph. This fact, which may be
seen as a one-sided discrete form of Weyl's law for surfaces, is consistent with the
intuition that at a small scale, bounded genus
 graphs behave like planar graphs. Finally, we note that the result in \cite{KLPT11} also
 applies to graphs in any fixed proper minor-closed family (where the genus $g$ is replaced with a parameter $h$ depending on the family), while the stronger bounds of Theorem~\ref{thm:main} cannot be extended to minor-closed classes, as we show by explicit examples in Remark~\ref{rem:mc}.

\subsection{Two-fold covers and their associated discrete Laplacians}

Let $M$ be a measured topological space, and denote by $\mu$ the measure on $M$.
A {\it 2-fold cover} of $M$ is a finite collection $\U = (U_v)_{v\in V}$, for a finite index set $V$, of open subsets $U_v$ of non-zero measure 
such that almost every point in $M$
is covered by exactly two subsets. To any 2-fold cover of  a measured space we associate a discrete Laplacian as follows:

We first form a graph $G=(V,E)$ on the set of vertices $V$ and with edges $\{u,v\} \in E$ for two vertices $u,v$ such that $\mu(U_v\cap U_u) \neq 0$. We define a weight function $\omega : E \rightarrow \mathbb R$ which to any edge $e=\{u,v\}$ of $G$, associates the weight $\omega(e) = \mu(U_u \cap U_v)$. The weighted valence $d^{\omega}_v$ of a vertex $v$ of $G$ is defined by 
  $$d^\omega_v = \sum_{u: u\sim v} \mu(U_u\cap U_v).$$
  
The discrete Laplacian associated to the 2-fold cover $\mathcal U$ denoted by
$\mathcal L_{\U}$ is the normalized graph Laplacian associated to the
weighted graph $(G,\omega)$. This is defined from the 
weighted Laplacian by normalizing using the weighted valence  (as in the previous section). Formally, define the weighted Laplacian $\Delta_{\mathcal U} : C(G) \rightarrow C(G)$ by 
$$\forall v\in V, \quad \Delta_{\mathcal U}(f)(v) = \sum_{u: u\sim v} \Bigl(f(v) -f(u)\Bigr) w(\{u,v\}),$$
for any $f\in C(G)$. Let $S_{\mathcal U}$ be the diagonal operator with entries the weighted valence $d^{\omega}_v$ of vertices $v\in V$, i.e., for any $f\in C(G)$,
\[\forall\, v\in V, \quad S_{\mathcal U}(f)(v) = d_v^\omega \, f(v).\]
Then we let $\mathcal L_{\U} :=S_{\mathcal U}^{-\frac 12} \Delta_{\mathcal U} S_{\mathcal U}^{-\frac 12}$. Denote by $\lambda_k(\mathcal L_\U)$ the $k$-th smallest eigenvalue of $\mathcal L_{\U}$.

\medskip

When $(M, \mu)$ carries a natural notion of Laplacian, it is possible to relate the eigenvalues of the Laplacian on $M$ to the eigenvalues of $\mathcal L_{\U}$ for any 2-fold cover $\U$. More precisely, let the measured space $(M, \mu)$ belong to any of the following three classes: 

\medskip

\begin{itemize}
 \item[($\mathscr C1$)] a smooth manifold with a smooth Riemannian metric $\g$, and $\mu$ the measure associated to the metric $\g$;

\item[($\mathscr C2$)]  a compact smooth surface  with a conformal class of smooth 
Riemannian metrics $\g$, and  $\mu$ a Radon measure absolutely continuous with respect to $\mu_\g$, c.f. Section~\ref{sec:prel};

\item[($\mathscr C3$)] a metric graph with $\mu$ the Lebesgue measure.
\end{itemize}
In any of the above cases, we can define a Laplacian on $(M, \mu)$ c.f.~Section~\ref{sec:prel} and Section~\ref{sec:metric}, and we denote by $\lambda_k(M,\mu)$, or simply $\lambda_k(M)$ if there is no risk of confusion, the eigenvalues for the corresponding Laplacian.

\medskip

Our  transfer principle is stated as follows. 

\begin{thm}\label{thm:2-fold} Let $(M, \mu)$ be a measured space as in $(\mathscr C1)$, $(\mathscr C2)$, or $(\mathscr C3)$ above.  
  Assume all the elements in a 2-fold cover $\U$ of $M$ have Neumann value
  at least $\eta$. Then for all positive integers $k$ we have:
$$
\lambda_k(\mathcal L_\U) \leq 2\frac{\lambda_{k}(M)}{\eta}.
$$
\end{thm}

The main difference with the classical versions of the transfer principle~\cite{B86, Br86, M05} is that we discretize the continuous Laplacian as a weighted normalized graph Laplacian instead of a
combinatorial one, which allows for a closer connection between the
two.  Our variant here uses a different notion of graph
approximation that involves particular weights. In addition, the above mentioned results take as input a partition of $M$, while our theorem is expressed in terms of  two-fold covers, which adds more flexibility. 

In order to prove Theorem~\ref{thm:main}, we apply the above theorem in the case where $(M, \mu)$ is a measured surface equipped with a conformal class of smooth 
Riemannian metrics $\g$. This version seems to
be required to get our Theorem~\ref{thm:main} on eigenvalues of bounded (geometric) genus
graphs.

\subsection{Organization of the paper}
The necessarily background on Laplacian eigenvalues in measured surfaces is recalled in Section~\ref{sec:surf}. 
The proof of Theorem~\ref{thm:2-fold} for measured surfaces (Case
($\mathscr C2$) among the above three cases) is given in
Section~\ref{sec:surf}. The proof in the two other cases is similar
and is thus omitted. Section~\ref{sec:surf} contains also the proof of
Theorem~\ref{thm:main}. In Section~\ref{sec:metric}, we apply Theorem~\ref{thm:2-fold} in the case of metric graphs with Lebesgue measure ($\mathscr C3$)),
  to obtain a uniform quantitative complement to a theorem of X. Faber~\cite{Faber} on the spectral convergence of finite graphs to
metric graphs. Moreover, we give in Section~\ref{sec:mesh} an algorithmic application of the above theorem to mesh partitioning in numerical analysis, 
generalizing the results of Miller-Teng-Thurston-Vavasis~\cite{MTTV98} and Spielman-Teng~\cite{ST07} to anisotropic meshes.

\section{Eigenvalues of bounded genus graphs}\label{sec:surf}
In this section we give the proofs of Theorem~\ref{thm:2-fold} and Theorem~\ref{thm:main}. We start by recalling the variational approach to study eigenvalue problems for 
surfaces with measures~\cite{Ko14}, which provides a setting to study eigenvalue problems for singular surfaces. 
This makes the statement of Theorem~\ref{thm:2-fold} precise in the case of a
measured metric surface.
 \subsection{Eigenvalues on measured surfaces}\label{sec:prel}
 Let $M$ be a smooth compact surface, possibly with boundary, which we suppose equipped with a smooth Riemannian metric 
$\g$. Denote by $\mu_\g$ the induced volume form on $M$. 
Let $\mu$ be a Radon measure on $M$ which we suppose 
absolutely continuous with respect to the measure $\mu_\g$.  
For a $C^\infty$-smooth function $f \in L^2(M,\mu)$, 
the Rayleigh quotient $R_{M_\g}(f, \mu)$ is defined by 
$$R_{M_\g}(f,\mu):= \frac{\int_M |\nabla_\g\, f|^2 d\mu_\g}{\int_M f^2 d\mu}.$$

The eigenvalues of the measured metric surface $(M_{\mathfrak g},\mu)$ are defined by the 
 variational formula:
\begin{equation}\label{eq:var1}
\lambda_k(M_\g, \mu) := \inf_{\Lambda_{k+1}}\, \sup_{f\in \Lambda^*_{k+1}} R_{M_\g}(f, \mu),
\end{equation}
where $\Lambda_{k+1} \subset L^2(M,\mu)$ varies over subspaces of dimension $k+1$ which consist
only of smooth functions on $M$, and 
$\Lambda^*_{k+1} = \Lambda_{k+1} \setminus \{0\}$. Note that in the
case $\mu =\mu_\g$, we recover the usual  
variational characterization of 
the eigenvalues of the Laplacian $\Delta_\g$ associated to the Riemannian surface $M_\g$.

To see the point of introducing this formalism, assume that the two metrics $\g$ and $\h$ on $M$ are 
conformally equivalent.  
From the conformal invariance of the Dirichlet integral, we
see that $R_{M_\g}(f,\mu) = R_{M_\h}(f, \mu)$. In particular, letting $\mu=\mu_\h$, we see that the spectra of
the metric $\h$ within the conformal class of $M_\g$ coincides with the spectra of
measured surface $(M_\g,\mu)$ for an appropriate Radon measure $\mu$. Now, if $\h$ is a
metric with conical singularities, it is a classical fact that $M$ is conformally equivalent to a constant 
curvature metric $\g$, the conformal factor being square integrable with respect to the
corresponding area form $\mu_\g$. 
Thus the framework of measured metric surfaces
allows in particular to define spectra of surfaces with conical singularities as the spectra of the measured metric surface 
$(M_\g, \mu_\h)$~\cite{Ko14}.

 Let $U$ be an open subset of $M$, and denote by $\overline U$ the topological closure of $U$ in $M$. 
 The Neumann value $\lambda(U)$ of $U$ is defined as the infimum of the Rayleigh ratio 
 $\int_U |\nabla_\g f|^2d\mu_\g/\int_U f^2 d\mu$ 
 over all smooth functions $f$ on $U$ which extend continuously to $\overline U$ and satisfy $\int_U f d\mu = 0$,
 \[\lambda(U):=\inf_{f:\,\,\,\int_U fd\mu=0} \,\,\frac{\int_U ||\nabla_\g f||^2 d\mu_\g}{\int_U f^2 d\mu}.\]

\subsection{Proof of Theorem~\ref{thm:2-fold}}
We suppose $M,\mu, \g$, and $\mu_g$ as above, and consider a 2-fold cover $\U = (U_v)_{v\in
  V}$ of $M$. Denote by $\eta$ the minimum of $\lambda(U_v)$ for $v\in V$. Let $G$ be 
  the associated weighted graph with vertex set $V$ and weight
matrix $W_{\U} = [\omega(\{u,v\})]_{u,v}$, where $\omega(\{u,v\}) = \mu(U_u\cap U_v)$ for $u\neq v$. Let
$\mathcal L_{\U}$ be the matrix of the associated normalized graph Laplacian. We have 
$\mathcal L_{\U}  =I - S_{\U}^{-1/2}W_{\U}S_{\U}^{-1/2}$, where the matrix $S_{\U}$
is diagonal with entries given by the weighted valences of the vertices $d^\omega_v = \sum_{u: u\sim v} \omega(\{u,v\})$.

\begin{proof}[Proof of Theorem~\ref{thm:2-fold}]
Let $v\in V$ and $f$ any smooth function on $M$.  By restricting $f$ to $U_v$ and
substracting the mean over $U_v$, we get:
\begin{eqnarray*}
\int_{U_v} ||\nabla_\g f||^2 \,d\mu_\g &\geq& \lambda(U_v) \int_{U_v} \bigl(f-\frac{1}{\mu(U_v)}\int_{U_v}fd\mu\bigr)^2\,d\mu\\
                                    &\geq& \eta\, \Bigl(\,\int_{U_v} f^2 \,d\mu-
                                    \frac{1}{\mu(U_v)}\,\bigl(\int_{U_v} f\, d\mu\bigr)^2\,\Bigr),\\
\end{eqnarray*}
Summing the last inequalities over $v\in V$ yields:
\begin{equation}\label{eq1}
\frac{2}{\eta}\int_M ||\nabla_\g f||^2 d\mu_\g \geq 2||f||_2^2 - \sum_v \frac{1}{\mu(U_v)}\bigl(\int_{U_v} fd\mu\bigr)^2,
\end{equation}
where the $L^2$ norm $||.||_2$ is with respect to the measure $\mu$. 
Denote by $\mathbf{1}_{U_v}$ the characteristic function of the open set $U_v$, and let 
$\phi_v = \mu(U_v)^{-1/2} \mathbf{1}_{U_v}$. Define $\Phi:
L^2(M)\to C(G)$, by 
$$\Phi(f)(v) := \int_M f\phi_v,$$
on any vertex $v$ of $G$. 
We see that the
quadratic form in $f$ in the right hand side of Equation~\eqref{eq1} is given by  $2||f||_2^2
- ||\Phi f||_2^2$. 

\medskip

Let $\epsilon>0$, and denote by $\Lambda_{k+1}^\epsilon$ a $(k+1)$-dimensional space of smooth functions on $M$ such that 
for any $f\in \Lambda_{k+1}\setminus \{0\}$, we have
\[ \frac{\int_M ||\nabla_\g f||^2d\mu_\g}{\int_M f^2 d\mu}\leq (1+\epsilon) \lambda_k(M).\]
Note that by the variational characterization of the eigenvalues
(see e.g.~\eqref{eq:var1} and~\eqref{eq:prin}), such a space exists.   
 For any $f\in \Lambda_{k+1}^\epsilon$, by inequality~\eqref{eq1}, we have:
$$
\frac{2(1+\epsilon)\lambda_{k}(M)}{\eta} ||f||_2^2 \geq
\frac{2}{\eta}\int_M ||\nabla_\g f||^2d\mu_\g \geq 2||f||_2^2 - ||\Phi f||_2^2.
$$ 
That is:
\begin{equation*}
||\Phi f||_2^2 \geq 2\Bigl(1-\frac{(1+\epsilon)\lambda_k(M)}{\eta}\Bigr) ||f||_2^2.
\end{equation*}
Let $\Phi^*$ denote the adjoint of the operator 
$\Phi: L^2(M) \rightarrow C(G)$.  From the variational characterization of the eigenvalues, this implies
that the compact self-adjoint operator $\Phi^* \Phi$ on $L^2(M)$ has at least $k+1$ eigenvalues
greater than or equal to $2(1-(1+\epsilon)\lambda_k(M)/\eta)$. We can assume that this latter quantity is positive, otherwise
there is nothing to prove since all the eigenvalues of the normalized
Laplacian are at most $2$. Since the non zero eigenvalues of $\Phi^*\Phi$
are the same as the non zero 
eigenvalues of $\Phi\Phi^*$, we thus deduce that 
\begin{equation}\label{eq2}
\lambda_k(2I - \Phi\Phi^*) \leq \frac{2(1+\epsilon)\lambda_k(M)}{\eta}.
\end{equation}

To conclude the proof, it suffices to notice that 
$$
[\Phi\Phi^*]_{u,v} = \int_M \phi_u\phi_v = \frac{\mu(U_u\cap U_v)}{(\mu(U_u)\mu(U_v))^{1/2}}.
$$
Because $\U$ is a 2-fold cover, the $v$-th entry of the diagonal matrix
$D_\U$ is equal to $\mu(U_v)$. We thus easily check that $2I - \Phi\Phi^* = \mathcal L_\U$. Therefore, inequality~\eqref{eq2} 
gives
\[\lambda_k(\mathcal L_\U) \leq \frac{2(1+\epsilon)\lambda_k(M)}{\eta}.\]
Since this holds for any $\epsilon>0$, the theorem follows.
\end{proof}

\subsection{Proof of Theorem~\ref{thm:main}}
We first give some background about graphs embedded in a surface, and refer to~\cite{MT01} for more details. We assume that all surfaces are compact, orientable and without boundary. An
\emph{embedding} of a graph $G$ in a surface $M$ is a drawing of $G$ on
$M$ so that all vertices of $G$ are distinct on $M$, and every edge of
$G$ form a simple arc on $M$ connecting its two endpoint
vertices. Interior of edges and vertices are assumed  to be pairwise
disjoint. A \emph{face} of and embedding, or simply a face of $G$ if
the embedding is clear from the context, is a connected component of the complementary of $G$ in $M$. 

An embedding is called \emph{cellular} if every face is homeomorphic to an open disk in $\R^2$.

 The genus $g(G)$ is the minimum integer $g$ such that $G$ has an embedding in a surface $M$ of genus $g$. The following result will allow us to suppose that a graph $G$
with a given genus $g(G)$ is embedded in a cellular way.

\begin{prop}[{\cite[Proposition 3.4.1]{MT01}}]
Every embedding of a connected graph $G$ in a surface of genus $g(G)$ is cellular. 
\end{prop}

Suppose from now on that  the connected graph $G$ is embedded in a cellular way in  a surface $M$ of genus $g$, so that every face $F$ is homeomorphic to an open disk $D_F$
in $\R^2$. The boundary of the face $F$ in $M$ is the image of the
boundary $\partial D_F \simeq S^1$ under a continuous map, which is
locally a homeomorphism away from the preimage of the vertices. We denote by $\mathcal F$ the set of all faces of
$G$. For any face $F \in \mathcal F$, we define a \emph{boundary walk} of $F$ to be any walk in the graph $G$ consisting of vertices and edges as
they are encountered when walking along the whole boundary of $F$, following the circle $\partial D_F$, and starting at some vertex. Note that some edges may appear more than once in a boundary walk. The \emph{degree} of a face $F \in \mathcal F$, denoted $\deg(F)$, is the number of edges on any boundary walk of
$F$.

\medskip

We define a new multigraph $\mathscr G = (\mathscr V, \mathscr E)$ embedded in $M$, and containing $G$ as an induced subgraph, by coning over boundary walks of faces as follows. The vertex set  $\mathscr V$ of $\mathscr G$ consists of the vertices in $G$ and a new vertex $v_F$ for each face $F$ of $ \mathcal F$, i.e., $\mathscr V =  V \sqcup \bigl\{v_F\bigr\}_{F\in \mathcal F}$.
For each face $F$ of $\mathcal F$, let $v_1, \dots, v_{\deg(F)}$ be
the vertices of $G$ which appear in this order in a boundary walk of
$F$. Note that a vertex might appear more than once.  
The edge set $\mathscr E$  of $\mathscr G$ consists of the edges in
$E$, and new edges $\{v_F, v_i\}, $ for ${i=1, \dots, \deg(F)}$. The embedding of $\mathscr G$ in $M$ is obtained in the following natural way: each face $F\in \mathcal F$ is homeomorphic to a disk $D_F$ in $\R^2$, and under this homeomorphism, the vertices $v_1, \dots, v_{\deg(F)}$ in the boundary walk of $F$, appear in this cyclic order on the boundary of $D_F$. 
Choose the center of $D_F$ as the image of $v_F$ and the rays from
$v_F$ to $v_i$ as the image of the edges $\{v_F, v_i\}$. We refer to
all the new edges $\{v_F, v_i\}$ added in the process 
 as {\it cone edges} of $\mathscr G$.

Call an embedding of a graph in $M$ a \emph{weak triangulation} if the
degree of any face of the embedding is three. We use this terminology
since it can happen that two different faces of the embedding share
more than one edge, in which case we do not have a triangulation. 

The embedded (multi)graph $\mathscr G$ constructed above has the following properties.
\begin{prop} \label{prop:basic1} The embedding of $\mathscr G$ in $M$
  is a weak triangulation, and each face of $\mathscr G$ is incident
  to exactly one edge of $G$. Moreover, for any vertex $v$ of $G$, we have $d^{\,\mathscr G}_v = 2d^G_v$, where $d_v^{\,\mathscr G}$ and $d^G_v$ denote the valence of $v$ in $\mathscr G$ and $G$, respectively. 
\end{prop}

\begin{proof}
By definition of the embedding, each face of $\mathscr G$ consists of two cone edges and an edge of $G$, which proves the first assertion.  To prove the second statement, let  $F_1, \dots, F_k \in \mathcal F$ be all the faces of $G$ which are incident to the vertex $v \in V$.  For each $i=1, \dots, k$, the number of edges $\{v_{F_i}, v\}$ in $\mathscr E$ is half the number of edges of $G$ in a boundary walk of $F_i$.  Each edge $e \in E$ incident to $v$ appears precisely  twice in the union of the edges of the boundary walks of $F_1, \dots, F_k$.   This shows that the total number of edges of $\mathscr E$ of the form $\{v_{F_i}, v\}$, for $i=1, \dots, k$, is equal to $d^G_v$, which proves the claim. \end{proof}

\begin{defi}[Open star] \rm For each vertex $v$ of $G$, we define the
  \emph{open star} of $v$ in $M$  with respect to the embedding of
  $\mathscr G$, or simply the open star of $v$, denoted by $\mathscr
  S_v$, as the interior of the union of all the faces of $\mathscr G$ which contain $v$ in their boundaries.  
\end{defi}

Let now $(M,\mathfrak g, \mu)$ be any surface as in Section~\ref{sec:prel},  so the measure $\mu$ is absolutely continuous with respect to the volume form $\mu_{\mathfrak g}$ of the smooth Riemannian metric $\mathfrak g$.  
\begin{prop}\label{prop:bas2} The open stars of vertices of $G$ form a 2-fold cover of $M$. 
\end{prop}
\begin{proof}
By Proposition~\ref{prop:basic1}, the boundary walk of each face of
$\mathscr G$ is a triangle which has exactly two vertices in $G$. It
follows that every point of $M\setminus \mathscr G$ appears in exactly
two open stars, which proves the claim by absolute continuity of $\mu$ with respect to $\mu_\mathfrak g$. 
\end{proof}

We now introduce a metric $\h$ on $M$ with conical singularities (and
will later assume $\mu  =\mu_\h$). For reasons that will soon become clearer, we set the length of each 
edge of $E\subset \mathscr E$ to be equal to one, and the length of
each cone edge in $\mathscr E$ to be $\cos(\pi/(2\dm))^{-1}/2$. We
equip $M$ with the natural metric $\h$ such the triangles have the
Euclidean metric induced by their edge lengths. Note that for any
triangle $T$ of $\mathscr G$, the angle of $T$ at any of its vertices
that belongs to the graph $G$ is equal to $\pi/(2\dm)$.  The metric $\h$ has only conical singularities, and we denote by $\lambda_k(M)=\lambda_k(M_\h)$, 
the eigenvalues of the surface $M$  as defined in Section~\ref{sec:prel}. Thus, $\lambda_k(M)$ is the $k$-th eigenvalue of the measured metric surface $(M_\g,\mu=\mu_\h)$, where $\g$ is a metric of constant curvature in the conformal class of $\h$.

Using Theorem~\ref{thm:2-fold} we can relate the eigenvalues
of $M$ to those of $G$. Denote by $\U$ the $2$-fold cover of $M$ given
by the open stars $\mathscr S_v$ of vertices of $V$,
c.f. Proposition~\ref{prop:bas2}. The intersection of two distinct open stars
 $\mathscr S_u$ and $\mathscr S_v$, for two vertices $u$ and $v$ of
 $G$, has non zero measure if and only if $u$ and $v$ are neighbors in
 $G$. Moreover, all non-empty intersections have the same
 measure, equal to the area of two triangles in $M$. Therefore, the normalized Laplacian $\mathcal L_{\U}$ equals
 the normalized Laplacian of $G$. Hence, in order to apply the transfer result Theorem~\ref{thm:2-fold}, we only need to lower bound the Neumann value of the open stars $\mathscr S_v$ of the vertices of $G$.

  We do so by 
again applying the transfer result to a specific 2-fold
cover of each open star $\mathscr S_v$, for $v\in V$.  Thanks to the
choice of edge lengths, the vertices of $G$ in $(M,\h)$  have non negative curvature.
It follows that by cutting $\mathscr S_v$ along an arbitrary cone edge, we can unfold $\mathscr S_v$  to the plane without overlap. Denote by $\mathscr S'_v$ the unfolded star of $v$ as isometrically embedded in $\R^2$. The cutting operation can only decrease the Neumann
value so it is sufficient to bound from below the Neumann value of the unfolded open subset
$\mathscr S'_v$ of $\R^2$. 

We call a \emph{kite} in $\mathscr S_v'$ the union of two triangles in $\mathscr S_v'$ which share an edge of $G$. So for any edge $e=\{v,u\} \in E$, there is a kite $K_e$, and the union of the kites $K_e$ for $e$ incident to $v$ is equal to the planar set $\mathscr S_v'$. For any edge $e\in E$ incident to $v$, the kite $K_e$ has two diagonals composed of the edge $e$ and the diagonal opposite to $e$, that we denote by $\mathrm{diag}^{\mathrm{op}}_e$. Cut $\mathscr S_v'$ along all the opposite diagonals $\mathrm{diag}^{\mathrm{op}}_e$ for $e \in E$ incident to $v$. This cuts $\mathscr S'_v$ into an open region $P_v$ with polygonal boundary containing the
vertex $v$, together with one triangle $T_e$ for each kite $K_e$ for $e\in E$ incident to $v$. 
Define 
$$\U_v := \bigl\{K_e\bigr\}_{e=\{v,u\}\in E} \cup \bigl\{T_e\bigr\}_{e=\{v,u\}\in E} \cup\bigl\{P_v\bigr\}.$$ 
The cover $\U_v$ of $\mathscr S'_v$ is a 2-fold cover. 

\begin{prop}Any $X \in \U_v$ is a convex set of diameter at most two.
\end{prop}
\begin{proof}  The triangle $T_e$, for $e=\{u,v\} \in E$, is obviously
  convex of diameter one, and so is the kite $K_e$. As for the region $P_v$, to prove the convexity of $P$, it will be enough to show that 
the angle of $P$ at $v$ is at most $\pi$. As we previously observed, by the choice of the edge lengths, all the triangles of $\mathscr G$ has angle $\pi/2\dm$ at any of their vertices which belong to $G$. The number of triangles of $\mathscr S'_v$ is at most $2d^G_v$, since $d_v^G \leq \dm$, it follows that the angle at $v$ of $P$ is at most $\pi$, and the convexity follows. The claim on the diameter follows from the fact that all the edges of $G$ have length equal to one. 
\end{proof}
Since any element of $\U_v$ is a planar convex sets  of 
diameter at most two,  there exists a universal constant $C_1>0$ such
that for any $X \in \U_v$, we have $\lambda(X) \geq C_1$ for
$\lambda(X)$ the Neumann value of $X$ \cite{PW60}.  For the 2-fold cover of $\mathscr S'_v$,  the non-zero element
$\mu_\h(X\cap Y)$ for $X\neq Y \in \U_v$ have the same value, equal to
the area of a triangle in $M$. Therefore, the normalized Laplacian
$\mathcal L_{\U_v}$ equals the normalized Laplacian of the graph whose
edges are the pairs $X\neq Y \in \U_v$ whose intersection has positive
measure. This graph $S^1_d$ is obtained from the star graph with $d$
edges by
inserting a new vertex in the middle of each edge. (Recall that a star graph with $d$ edges has a central vertex connected to $d$ other vertices.) The Neumann value
of $S^1_d$ is lower bounded by an absolute constant $C_2>0$ independent of $d$. Hence, applying the transfer theorem~\ref{thm:2-fold} to the 2-fold cover $\U_v$, it follows that there exists a universal constant $C_3 = C_1.C_2/2>0$ such that  the Neumann value of $\mathscr S_v'$ is bounded from below by $C_3$, i.e., $\lambda(\mathscr S'_v) \geq  C_3$. This gives $\lambda(\mathscr S_v) \geq  C_3.$

We get from these observations, and Theorem~\ref{thm:2-fold} applied to the 2-fold cover of $\U$, that 
\begin{align}\label{eq*}
\lambda_k^{\mathrm{nr}} = \lambda_k(\mathcal L_\U) \leq 2 \frac{\lambda_k(M)}{C_3}.
\end{align}

A result of Hassannezhad \cite{H11} states now that there is a universal constant $A$
such that for each $k$:
\begin{align}\label{eq**}
\lambda_k(M)\mu(M) \leq A(g+k).
\end{align}
Note that this result is not explicitly stated in the framework of measured metric surfaces in~\cite{H11}, 
however the proof given in~\cite{H11} works also in this setting and gives the above statement. Putting Equations~\eqref{eq*} and \eqref{eq**} together, and observing
that $\mu(M) \geq C_4  n/\dm$ for some constant $C_4$, we conclude
that for $C = 2A/(C_3C_4)$, we have
\[\lambda^{\mathrm{nr}} (G) \leq C \frac{\dm (g+k)}{n},  \]
which is the statement of Theorem~\ref{thm:main}. \qed

\begin{remark}\label{rem:tightness}\rm 
It is shown in \cite{CE03} that
 for large $g$, there are area one and genus $g$ Riemannian surfaces $S$ 
 with $$\lambda_k(S) \geq \frac{4\pi}{5} (g-1) + 8\pi (k-1) -\epsilon$$
for any $\epsilon>0$. Now, the classical Brooks-Burger method implies the existence of a bounded degree
genus $g$ graph $G$ with $n$ vertices such that $\lambda_k(G) \geq C
\lambda_k(S)/n$. Hence, at least for large enough $n$ and $g$, there
are graphs whose eigenvalues match the behaviour of the estimate in Theorem~\ref{thm:main}.
\end{remark}

 \begin{remark}\label{rem:mc}\rm The following example shows that the strong estimates as in Theorem~\ref{thm:main} cannot hold for more general classes of graphs closed under taking minor. 

Recall that the Cartesian product  $G_1 \square G_2$ of two graphs $G_1 = (V_1, E_1)$ and $G_2=(V_2, E_2)$ has vertex set $V_1\times V_2$ and there is an edge between $(v_1,v_2) $ and 
$(u_1,u_2)$ in $V_1\times V_2$ if either $u_1=v_1$ and $\{u_2,v_2\}\in E_2$, or $u_2=v_2$ and $\{u_1,v_1\} \in E_1$. The Laplacian eigenvalues of $G_1 \square G_2$ are of the form $\lambda_i(G_1) + \lambda_j(G_2)$ for $i=1,\dots, |V_1|$ and $j=1,\dots, |V_2|$. 

Let $d$ be a fixed large enough integer, and for any  $\ell \in \mathbb N$, consider the Cartesian product  $C_{2l} \square G$ of a cycle $C_{2\ell}$ of length $2l$ with a $d$-regular
 graph $G$ on $t$ vertices, for an  integer $t\in \mathbb N$.   
 
 For any fixed $t\in \mathbb N$, we get in this way a family of
 graphs by varying $\ell$ and $G$. All these graph are of treewidth bounded by some $f(t)$ for a (linear) function $f$ of $t$. Bounded treewidth graphs form a minor-closed family, so all these graphs belong to a fixed proper minor-closed family $\mathcal F_t$. For $G$ a random $d$-regular graph on $t$ vertices, and for the ${l}^{\mathrm{th}}$
 eigenvalue of $C_{2l} \square G \in \mathcal F_t$, for $l\in \mathbb
 N$, we have $\lambda_l (C_{2\ell}\square G) =
 \Omega(\frac{tl}{|C_{2l}\square G|})$ with high probability as $t$
 tends to infinity. This shows that there do not exist in general
 constants $h = h(\mathcal F_t)$ and $C = C(\mathcal F_t)$ associated
 to $\mathcal F_t$ ensuring that the inequality  $\lambda_k(G) \leq C \, \dm^2 \, (g_t+ k)/n$ hold  for any  graph $G \in \mathcal F_t$ on $n$ vertices, and for any $k \in \mathbb N$ (unlike what happens for
 the class of bounded genus graphs). In particular, the strong estimates as in Theorem~\ref{thm:main} cannot hold for general minor-closed classes of graphs. 
 \end{remark}

\section{Eigenvalues of the Laplacian on metric graphs}\label{sec:metric}
We briefly review the basic definitions concerning the spectral theory of metric graphs, and refer e.g. to~\cite{BR07, Zhang} for more details.

Let $G=(V,E)$ be a finite connected graph and let $\ell: E \rightarrow \mathbb R_{>0}$ be a (length) function on the edges of $G$. The length of $e$ is denoted by $\ell_e$. We define the \textit{metric realization} of $(G, \ell)$ as follows: for each edge $e =uv$ of $G$ take a closed
interval $I_e\subset \mathbb R$ of length $\ell_e$, and a surjection $\pi_e: \partial I_e \rightarrow \{u,v\}$ (which identifies the two extremities of $I_e$ with the vertices of $G$ in $e$).  Define the topological space (with the quotient topology)
$$\Gamma : = (V \sqcup \bigsqcup_e I_e) /\bigl\{\,x=\pi_e(x)\,\,\, \forall e\in E \,\,\&\,\, x\in \partial I_e \,\bigr\}.$$
The space $\Gamma$ has a natural metric, the shortest path metric induced by piecewise isometric paths between points, see e.g.~\cite{BR07}.  We call a \emph{metric graph} any metric space $\Gamma$ isometric to a metric realization of a pair $(G,\ell)$, as above. The pair $(G,\ell)$ is called a  \emph{model} of $\Gamma$; when $G$ is a simple graph, the model is called simple.  Note that there are plenty of models for a metric graph $\Gamma$, e.g. any finite subset of points of $\Gamma$ can be part of a simple model of $\Gamma$.

For any point $p\in \Gamma$, we denote by $T^1_p\Gamma$ the set of
\emph{unit tangent vectors} to $\Gamma$ at $p$.  For an interval $I =
[a,b]$, in $\R$, we define $T^1_aI = \{\vec 1\}$, with $\vec 1$ the
unit vector in $\R$. For a metric graph $\Gamma$ and a point $p\in
\Gamma$, let $(G, \ell)$ be a simple model of $\Gamma$ with $p\in
V(G)$, and let $e_1, \dots, e_d$ be the edges of $G$ incident to
$v$. Define $T^1_p\Gamma$ as the set of all unit tangent vectors at
$p$ of the intervals $I_{e_j}$, as above. Let $\vec u \in T^1_p\Gamma$ be a unit tangent vector, and let $I=I_e$ be the corresponding interval (corresponding to the edge $e$ of a simple graph model $(G,\ell)$). For $\epsilon>0$ sufficiently small, we denote by $p+ \epsilon \vec u$
the unique point in $I$ at distance $\epsilon$ from $p$ on $I$. A function $f: \Gamma \rightarrow \R$ is \emph{piecewise smooth} if there exists a simple graph model $G=(V,E)$ of $\Gamma$ such that the restriction of $f$ to the intervals $I_e$, for $e\in E$, are of class  $C^2$.   The space of piecewise smooth function on $\Gamma$ is denoted by
$S(\Gamma)$. Let $f: \Gamma \rightarrow \R$ be a piecewise smooth function on a metric graph $\Gamma$. Let $p\in \Gamma$ and $\vec u \in T^1_x\Gamma$ a unit tangent vector to at $x$. The \emph{(outgoing) 
 slope} of $f$ along $\vec u$ denoted by $d_{\vec u}(f)$ is  defined by
 \[d_{\vec u}(f) := \lim_{\epsilon \to 0^+} \frac{f(p+ \epsilon \vec u) - f(p)}{\epsilon}.\]
For a point $p\in \Gamma$, we define $\sigma_p$ as the sum of the slopes of $f$ along unit tangents:
$$\sigma_p := \sum_{\vec u \ni T^1_p\Gamma} d_{\vec u}f( p ),$$
Note that for all but at most a finite number of points $p\in \Gamma$, we have $\sigma_p =0$.  A metric graph $\Gamma$ has a natural Lebesgue measure denoted by $dx$. The Laplacian of $\Gamma$ is the (measure valued) operator $\Delta$ on $\Gamma$ which to a function $f\in S(\Gamma)$ associates the measure 
$$\Delta(f):= - f''dx -\sum_{p \in \Gamma} \sigma_p \delta_p.$$

Define the \emph{Zhang space} $\mathrm{Zh}(\Gamma)$ as the space of all functions $f \in S(\Gamma)$ such that $f'' \in L^1(\Gamma, dx)$. 
The inner product $(\,,)$ and the Dirichlet pairing $(\,,)_{\mathrm{Dir}}$ on $\mathrm{Zh}(\Gamma)$ are defined by
$$\forall f,g\in \mathrm{Zh}(\Gamma), \qquad (f,g) := \int_{\Gamma} fg\,\, dx, \,\, \, \textrm{ and }$$ 
$$(f,g)_{\mathrm{Dir}} := \int_{\Gamma} f \Delta(g) = \int_{\Gamma}g \Delta(f) = \int_{\Gamma} f'g' \,dx =(f',g').$$
 
 A function $f$ in $\mathrm{Zh}(\Gamma)$ is an
  eigenfunction of the Laplacian on $\Gamma$ with eigenvalue $\lambda$ if for any
  function $g\in \Zh(\Gamma)$, we have $(f,g)_{\mathrm{Dir}} = \lambda
  (f,g)$. The eigenvalues of $\Delta$ are all nonnegative and,
  assuming $\Gamma$ is connected, they form a
  discrete subset $0=\lambda_0(\Gamma) <\lambda_1(\Gamma)<\lambda_2(\Gamma) <\dots< \lambda_n(\Gamma)<\dots$ of $\mathbb R$. 
In addition, $\lambda_k(\Gamma)$ has the following (usual) variational characterization:
\begin{equation}\label{eq:prin}
\lambda_k(\Gamma) = \inf_{\substack{\Lambda_{k+1} \subset \mathrm{Zh}(\Gamma)\\ \dim(\Lambda_{k+1}) = k+1}}\,\, \sup_{f\in \Lambda_{k+1}} \frac{\,\,\,(f,f)_{\mathrm{Dir}}}{(f,f)}
\end{equation}
\begin{defi}[Dilation of a metric graph]\rm
Let $\Gamma$ be a metric graph with a simple graph model $(G, \ell)$, and $\beta\in \mathbb R_{>0}$. The metric graph $\beta \Gamma$ is defined as the metric realization of the pair $(G, \beta \ell)$.
\end{defi}
The following proposition is straightforward, see e.g.~\cite{BR07}.
\begin{prop} Let $\Gamma$ be a metric graph and $\beta >0$ a real.
For any integer $k\geq 0$, we have $\lambda_k(\beta \Gamma) = \frac 1{\beta^2}\lambda_k(\Gamma)$. 
\end{prop}

By a \emph{metric star} $\mathscr S$ we mean the metric realization of a pair $(S_d, \ell)$ with $S_d =K_{1,d}$ a star graph of arbitrary valence $d$, and $\ell$ a length function on $E(S_d)$. For such a metric star, define  $\ell_{\max}(\mathscr S) := \max_{e\in E(S_d)} \ell(e)$.

\begin{lemma}\label{lem:star}
For any metric star $\mathscr S$, we have:
$$
\lambda_1(\mathscr S) \geq \frac{\pi^2}{4\ell^{\,2}_{\max} (\mathscr S)}.
$$
\end{lemma}

\begin{proof} Assume that $\mathscr S$ is the metric realization of a pair $(S_d, \ell)$ with $d\in \mathbb N$.   We adapt the argument in ~\cite[Example 3]{Fr} to the case where the branches
of $\mathscr S$ have non-necessary equal lengths.  Let us parametrize each each edge $e$ of $S_d$ with the interval $[0, \ell_e]$ starting from the leaf vertex towards the central vertex of $S_d$. In this
parametrization, an 
eigenfunction $\phi$ of the Laplacian, with corresponding eigenvalue $\lambda$, must be of the form $a_e
\cos(\sqrt{\lambda} x_e)$, where $x_e$ is the length parameter of
the edge $e$ in $\mathscr S$, for $e\in E(S_d)$. This follows in particular from the fact that
the slope of an eigenfunction must be zero at leaves. Now let $a$ be
the value of the eigenfunction $\phi$ at the center of $\mathscr S$. If $a=0$, we get that $\sqrt{\lambda} \ell_e \in
\pi/2 + \mathbb{N}$, for any edge $e$, which implies the claim. If $a$ is non zero, then
we use the fact that at the center of $\mathscr S$, the sum of the (out-going) slopes of $\phi$ along
the branches must be zero~\cite[Proposition 15.1]{BR07}, which gives
$$
\sum_{e\in E(S_d)} a_e \sin(\sqrt{\lambda} \ell_e) = 0.
$$
Since $a_e\cos(\sqrt{\lambda} \ell_e) = a$ for any edge $e$ of $S_d$, this implies
$$
\sum_{e\in E(S_d)} \tan(\sqrt{\lambda} \ell_e) = 0,
$$
and so, again, at least one of the arguments in the tangents must be at least
$\pi/2$, and the lemma follows.
\end{proof}

For a simple graph $G$ and a vertex $v\in V(G)$, we denote by $\Sigma_G(v)$ the star subgraph of $G$ with central vertex $v$ and with the edge set all the incident edges to $v$.  Let $(G,\ell)$ be a simple graph model of a metric graph $\Gamma$. For any $v\in V(G)$, we define the \emph{metric star with center $v$} (with respect to $G$)  of $\Gamma$  denoted by $\mathscr S_G(v)$, or simply $\mathscr S_v$ if there is no risk of confusion, as the subset of $\Gamma$ isometric to the metric realization of $\Sigma_G(v)$ with length function given by $\ell$. 
Denote by $\ell_{\max,G}$ the maximum length of edges in $G$, and note that $\ell_{\max,G} = \max_{v\in V(G)}\{\ell_{\max}(\mathscr S_v)\}$.

Given a simple graph model $(G,\ell)$ of a metric graph $\Gamma$, the family of all the metric stars $\mathscr S_v$, for $v\in V(G)$,
forms a 2-fold cover $\mathcal S$ of $\Gamma$. Denote by
$\lambda_k^{nr}(G,\ell)$ the $k$-th eigenvalue of $\mathcal{L_S}$. 

Lemma~\ref{lem:star} together with Theorem~\ref{thm:2-fold} yields the
following bound:

\begin{thm}\label{cor1} 
Let $\Gamma$ be a metric graph with a simple graph model $(G, \ell)$. For any $k\in \mathbb N$, we have
$$\lambda_k(\Gamma) \geq \frac{\pi^2}{8\ell_{\max,G}^2}\,\lambda_k^{nr}(G,\ell) .$$
\end{thm}
 We now show that under certain natural conditions, it is possible to achieve eigenvalue
  upper bounds closely matching the lower bounds of the above corollary.
For a simple graph model $(G,\ell)$ of $\Gamma$ denote by $\ell_{\min,G}$ the minimum length of edges $e$ in $E(G)$. 
\begin{defi} \rm 
 A simple graph model of a metric graph $\Gamma$ is called \emph{length-balanced} if for any edge $e\in E(G)$, we have $\ell_e\leq 2 \ell_{\min,G}$. 
\end{defi}
We have the following theorem. 
\begin{thm}\label{thm:main2} There are absolute constants $c_1,c_2>0$ such that for any length-balanced simple graph model $(G, \ell)$ of $\Gamma$ on $n$ vertices, and for any non-negative integer $k\leq n-1$, we have 
$$\frac {c_2}{d_{\max}} {\ell^2_{\min,G}} \, \lambda_{k}(\Gamma) \leq  \lambda_k^{\nr}(G,\ell) \leq c_1 \ell^2_{\min,G} \, \lambda_k(\Gamma).$$
\end{thm}
Before giving the proof, we state an interesting corollary of the above theorem. We first need the following definition. 
\begin{defi} \rm Let $\Gamma$ be a metric graph. Define $\ell_{\min}$ as the supremum of $\ell_{\min,G}$ over all length-balanced simple graph models $(G, \ell)$ of $\Gamma$.
\end{defi}

 It is easy to see that there is a length-balanced simple graph model $G$ of $\Gamma$ such that $\ell_{\min}=\ell_{\min,G}$. For such a simple graph model $(G,\ell)$, define the model $(G_k,\ell)$ as the $k$-th subdivision of $G$ where each edge $e$ is subdivided into $k$ edges of equal lengths $\ell_e/k$. Note that $G_k$ is length-balanced, has at least $k+1$ vertices, and has minimum edge length equal to $\ell_{\min}/k$. Thus as a consequence of Theorem~\ref{thm:main2}, we get
\begin{cor} \label{thm:main3} With the notations as above, there are absolute constants $c_1$ and $c_2$ such that 
for any metric graph $\Gamma$, we have 
$$ \frac {c_2}{d_{\max}} \ell_{\min}^2 \lambda_k(\Gamma) \leq k^2 \lambda_k^{\nr}(G_k,\ell) \leq c_1 \ell_{\min}^2 \lambda_k(\Gamma).$$
\end{cor}

Our results, especially corollary~\ref{thm:main3} above, should be viewed as a quantitative complement to a theorem of X. Faber~\cite{Faber} on the spectral convergence of finite graphs to
metric graphs, in the sense that they provide uniform upper and lower
bounds on the eigenvalues of $\Gamma$ in terms of eigenvalues of
simple graph models of $\Gamma$. 

\begin{proof}[Proof of Theorem~\ref{thm:main2}] First note that since $\lambda_k(\beta \Gamma) = \frac 1{\beta^2}\lambda_k(\Gamma)$ and since $\lambda^{\nr}_k(G, \beta\ell) = \lambda_k^{\nr}(G, \ell)$, by the very definition, it will be enough to prove the theorem for $\ell_{\min} =1$.  

The right hand side inequality follows from Theorem~\ref{cor1}, 
and the well-balanced property of the simple graph model $G$ of $\Gamma$. We now prove the other inequality, namely the existence of $c_2$ such that for any $k\leq n-1$, ${c_2} \, \lambda_{k}(\Gamma) \leq d_{\max}\lambda_k^{\nr}(G,\ell)$ (still under the assumption that $\ell_{\min}=1$ and the length-balanced property of the model $(G,\ell)$).  Since the lengths of all edges are between 1 and 2, we get $\lambda_k^{\nr}(G, \ell) \geq \frac 1{2d_{\max}} \lambda_k(G)$. Indeed, letting $g= D_{\mathcal S}^{1/2}f$, we have the following expression for the Rayleigh quotient 
$$\frac{(g, \mathcal L_{\mathcal S}g)}{(g,g)} = \frac{\sum_{e=\{u,v\}\in E} \ell(e) (f(u)-f(v))^2}{\sum_{v} d^\ell_v f(v)^2} \geq \frac{1}{2d_{\max}} \frac{\sum_{e=\{u,v\}\in E} (f(u)-f(v))^2}{\sum_{v} f(v)^2}$$
(where $d^\ell_v = \sum_{u: u\sim v} \ell(\{u,v\})$), which using the variational characterization of the eigenvalues proves the claim. So it will be enough to show the existence of a constant $c_2'$ such that 
$$c_2' \lambda_k(\Gamma) \leq \lambda_k(G).$$

Consider $W_{k+1}$ the vector space of dimension $k+1$ generated by the first $k+1$ 
eigenfunctions $g_0, \dots, g_k \in C(G)$ associated to $\lambda_i(G)$, for $i=0,\dots, k$. 
Note that in particular
$$\lambda_k(G) \geq  \sum_{\substack{u,v \in V(G)\\ u\sim v}}\frac{(g(u) -g(v))^2}{\sum_v g(v)^2}$$ 
for any $g\in W_{k+1} \setminus \{0\}$.  We construct an injective linear
map $\Psi: C(G) \rightarrow \Zh(\Gamma)$ such that for any  $g \in C(G) \setminus \{0\}$, we have 
$$\frac{\,\,\,(\Psi(g),\Psi(g))_{\mathrm{Dir}}}{(\Psi(g),\Psi(g))} \leq  8  \sum_{\substack{u,v \in V(G)\\ u\sim v}}\frac{(g(u) -g(v))^2}{\sum_v g(v)^2}.$$
Applying the variational characterization of $\lambda_k(\Gamma)$, given in Equation~\eqref{eq:prin}, to the test space $\Psi(W_k)$, for $k\leq n-1$, will then give the result. 

\medskip 

Consider an edge $e = \{u,v\}$ of $G$, and denote by $u_e$ and $v_e$ the two points at distance 
$\frac 1{4d_u}$ and $\frac 1{4d_v}$ from $u$ and $v$ on $e$, respectively, where $d_u$ and $d_v$ denote the valence of the vertices $u$ and $v$ in $G$, respectively.  Note that the length of each segment $[u_e, v_e]$ in $\Gamma$ is at least $\frac 12$. 

For any vertex $v$ of $G$, denote by $B_v$ the union of all segments $[v,v_e]$ on the edges $e$ adjacent to $v$ in $G$ (i.e., $B_v$ is the ball of radius $\frac 1{4d_v}$ around $v$ in $\Gamma$). For any function $g \in C(G)$, 
defined on the set of vertices of $G$, let $\Psi(g)$ be the function on $\Gamma$ which takes value equal to $g(v)$ on each ball $B_v$, and which is affine linear of slope $(g(v) - g(u))/\ell([u_e,v_e])$ on each segment $[u_e,v_e]$, 
for any edge $e\in E(G)$. Obviously, $\Psi$ is an injective linear map from $C(G)$ to $\Zh(\Gamma)$. 

Let now $g\in C(G) \setminus \{0\}$ and denote $f = \Psi(g)$. We have 
$$(f, f)_{\mathrm{Dir}} = \int_{\Gamma} f'^2 dx = \sum_{e=\{u,v\}\in E(G)} 
\frac 1{\ell([u_e,v_e])}(g(u)-g(v))^2 \leq 2\sum_{\{u,v\}\in E(G)} (g(u)-g(v))^2.$$

Denote by $B$ the union $\cup_{v\in V(G)} B_v$. Since each ball $B_v$ has  total length equal
to $1/4$, we have 
$$\int_\Gamma f^2 dx \geq \int_{B}f^2 dx = \frac 14 \sum_{v\in V(G)} g(v)^2.$$

It thus follows from the two above estimates that for any $g\in C(G)\setminus \{0\}$, we have 
$$\frac{\,\,\,(\Psi(g),\Psi(g))_{\mathrm{Dir}}}{(\Psi(g),\Psi(g))} \leq   8 \sum_{\substack{u,v \in V(G)\\ u\sim v}}\frac{(g(u) -g(v))^2}{\sum_v g(v)^2},$$ and the theorem follows. 
\end{proof}

\section{Anisotropic mesh partitioning}\label{sec:mesh}
In this final section we discuss a practical application of our transfer theorem to the mesh 
partitioning problem in scientific computing. Parallelizing finite elements computations requires to split the base
mesh in such a way that communication between different pieces is
minimized. This is naturally formalized as a (possibly multi-way)
sparsest cut problem, which we may want to solve using spectral
clustering. Guarantees for such methods in this setting were proved by
Miller-Teng-Thurston-Vavasis and Spielman-Teng~\cite{MTTV98,ST07}. 
More precisely, these papers show that spectral partitioning provides good cuts for 
meshes in $d$-dimensional Euclidean space provided that all $d$-simplices in the mesh are well-shaped,
 i.e. not too far from being equilateral.

 It is not hard to
design a $2$-fold cover of a general mesh such that our transfer result
provides guarantees for spectral clustering applied to anisotropic
meshes. Specifically, let $T$ be a triangulation of a domain $D\subset
\R^d$. Performing
a barycentric subdivision of all $d$-simplices gives a triangulation
$T'$. For a $d$-simplex $\sigma$ of $T$, let now $U_\sigma$ be the
interior of the
union of $\sigma$ with the $d+1$ $d$-simplices of $T'$ that share a
facet with $\sigma$. The collection of $U_\sigma$ forms a $2$-fold cover $\U$ of the domain,
and the corresponding Laplacian $\mathcal L_\U$ is defined using weights
$w_{\sigma_1,\sigma_2}$ that are proportional to the sum of the
volumes of $\sigma_1$ and $\sigma_2$. Hence, assuming that neighboring
$d$-simplices in $T$ have volumes within a ratio of $\kappa>1$, we see that the
eigenvalues of $\mathcal L_\U$ and those of the normalized Laplacian of the dual
graph of $T$ are also within a ratio of $\kappa$. 

\begin{prop}\label{fiedlermesh}
The Neumann value of $U_{\sigma}$ is at least $C^{-1}\kappa^{-1}\epsilon^{-2}$
for some universal constant $C>0$, where
$\epsilon$ is the maximum diameter of simplices in $T$.
\end{prop}
\begin{proof}
Let $\tau_i$, $i=1\dots d+1$, be the $d$-simplices in $T'$ that share
a facet with $\sigma$, and $\sigma_i$ be the
$d$-simplex in $T'$ that is included in $\sigma$ and shares a facet
with $\tau_i$. The interiors of $\sigma$, $\tau_i$, and of $\tau_i
\cup \sigma_i$ form a 2-fold cover of $U_{\sigma}$. The entries of the
corresponding Laplacian are within a factor $\kappa$ of the those of
the normalized Laplacian of the intersection graph of the elements of
the cover, which is a once subdivided star graph. Such a star graph
has Fiedler value lower bounded by a constant. Now each element in the
cover is a convex set with diameter at most $2\epsilon$, so by
\cite{PW60} their Neumann value is lower bounded by a constant times
$\epsilon^{-2}$. The claim then follows from theorem \ref{thm:2-fold}.
\end{proof}

Therefore, Theorem~\ref{thm:2-fold} applied to the cover $\U$ yields that the Fiedler value of the dual graph of $T$ is
at most $2C\kappa^2 \lambda_1(D) \epsilon^2$. By Cheeger's inequality,
a suitable spectral partitioning algorithm gives a balanced cut of size at most
$\kappa C'\sqrt{\lambda_1(D)}/\epsilon$, for some constant $C'$. We note that if $d$-simplices
in $T$ are nearly equilateral, then $\epsilon \simeq
(\vol(D)/n)^{1/d}$, where $n$ is the number of simplices in $T$. Hence
in this case 
we recover the $n^{1/d}$ behaviour proved in \cite{MTTV98,ST07} for the size
of the cut, since the assumption that simplices are well-shaped
implies an upper bound on $\kappa$. However, the methods used in those works do not seem to
apply to the case of general anisotropic meshes.

\medskip

\thanks{{\it Acknowledgments}: This work was started during a stay of the first named author 
at Laboratoire J. A.
Dieudonn\'e at Universit\'e de Nice Sophia-Antipolis, and pursued during another 
visit at INRIA Sophia-Antipolis. He thanks both these institutions, specially Philippe Maisonobe and Jean-Daniel Boissonnat, for their support. 

This research has
been partially supported by the European Research Council under
Advanced Grant 339025 GUDHI (Geometry Understanding in High Dimensions).
}

\end{document}